\documentclass[12pt]{article}
\usepackage{amsmath,amscd,amsthm,amssymb}

\def\dual{\,^{^{\complement}}\!}
\def\Rn{{\mathbb{R}^n}}

\def\i{\infty}
\def\a {\alpha}

 \newtheorem{thm}{Theorem}[section]
 \newtheorem{cor}[thm]{Corollary}
 \newtheorem{lem}[thm]{Lemma}
 
 \theoremstyle{definition}
 \newtheorem{defn}[thm]{Definition}
 \theoremstyle{remark}
 \newtheorem{rem}[thm]{Remark}
 
 \numberwithin{equation}{section}

\usepackage{amssymb}
\usepackage{color}

\def\R{{\mathbb{R}}}
\def\Rn{{\mathbb{R}^n}}
\def\a {\alpha}
\def\i{\infty}

\def\L1loc{L_{\Phi}^{\rm loc}(\Rn)}

\def\dual{\,^{^{\complement}}\!}

\newcommand{\es}{\mathop{\rm ess \; inf}\limits}
\newcommand{\lb}{\lambda}
\newcommand{\vi}{\varphi}

\begin{document}

\begin{center}
\LARGE A characterization for Adams-type boundedness of the fractional maximal operator on generalized Orlicz-Morrey spaces
\end{center}

\

\centerline{\large Fatih Deringoz$^{a,}$\footnote{
Corresponding author.
\\
E-mail adresses: deringoz@hotmail.com (F. Deringoz), vagif@guliyev.com (V.S. Guliyev),  sabhasanov@gmail.com (S.G. Hasanov).}, Vagif S. Guliyev$^{a,b}$, Sabir G. Hasanov$^{c}$ }

\

\centerline{$^{a}$\it Department of Mathematics, Ahi Evran University, Kirsehir, Turkey}

\centerline{$^{b}$\it Institute of Mathematics and Mechanics of NAS of Azerbaijan, Baku, Azerbaijan}

\centerline{$^{c}$\it Ganja State University, Ganja, Azerbaijan}

\

\begin{abstract}
In the present paper, we shall give a characterization for weak/strong
Adams type boundedness of the fractional maximal operator on generalized Orlicz-Morrey spaces.
\end{abstract}

\



\

\section{Introduction}

As is well known,  Morrey  spaces are widely used to investigate the local behavior of solutions to second order elliptic partial differential
equations (PDE). Recall that the classical Morrey spaces $\mathcal{M}^{p,\lb}(\Rn)$ are defined by
\begin{equation*}
\mathcal{M}^{p,\lb}(\Rn) = \Big\{ f \in L^p_{\rm loc}(\Rn) : \left\| f\right\|_{\mathcal{M}^{p,\lb}}
= \sup_{x \in \Rn, \; r>0 } r^{-\frac{\lb}{p}} \|f\|_{L^p (B(x,r))} < \infty  \Big\},
\end{equation*}
where $0 \le \lambda \le  n,$ $1\le p < \infty $. $\mathcal{M}^{p,\lb}(\Rn)$ was an expansion of $L^p(\Rn)$ in the sense that $\mathcal{M}^{p,0}(\Rn)=L^p(\Rn)$ and $\mathcal{M}^{p,n}(\Rn)=L^\infty(\Rn)$.
Here and everywhere in the sequel $B(x,r)$ is the ball in $\Rn$ of radius $r$ centered
 at $x$ and  $|B(x,r)|=v_n r^n$ is its Lebesgue measure, where $v_n$ is the volume of the unit ball in $\Rn$.

By $W\mathcal{M}^{p,\lambda}(\Rn)$ we denote the weak Morrey space defined as the set of functions $f$  in the local weak space $WL^{p}_{\rm loc}(\Rn) $ for which
$$
\left\| f\right\|_{W\mathcal{M}^{p,\lambda }} = \sup_{x\in  \Rn, \; r>0} r^{-\frac{\lambda}{p}}
\|f\|_{WL^{p}(B(x,r))} <\infty.
$$
The spaces $\mathcal{M}^{p,\vi}(\Rn)$ defined by the norm
\begin{equation*}
\left\| f\right\|_{\mathcal{M}^{p,\vi}}
= \sup_{x \in \Rn, \; r>0 } \varphi(x,r)^{-1}\,|B(x,r)|^{-\frac{1}{p}} \|f\|_{L^p (B(x,r))}
\end{equation*}
with a  function $\vi$ positive and measurable on $\Rn\times (0,\infty )$ are known as generalized Morrey spaces. Also by $W\mathcal{M}^{p,\vi}(\Rn)$ we denote the weak generalized Morrey space of all functions $f\in  WL^{p}_{\rm loc}(\Rn) $ for which
$$
\left\| f\right\|_{W\mathcal{M}^{p,\vi}}
= \sup_{x \in \Rn, \; r>0 } \varphi(x,r)^{-1}\,|B(x,r)|^{-\frac{1}{p}} \|f\|_{WL^p (B(x,r))}<\infty.
$$
Note that, in the case $\varphi(x,r)=r^{\frac{\lambda-n}{p}}$, we get the classical Morrey space $\mathcal{M}^{p,\lambda}(\Rn)$ from generalized Morrey space $\mathcal{M}^{p,\vi}(\Rn)$.

The Orlicz space was first introduced by Orlicz in \cite{456s, 456sb} as generalizations of Lebesgue spaces $L^p$. Since then
this space has been one of important functional frames in the mathematical analysis, and
especially in real and harmonic analysis. Orlicz space is also an appropriate substitute for $L^1$ space
when $L^1$ space does not work. For example,
the Hardy-Littlewood maximal operator
\begin{equation}\label{maxoperator}
Mf(x)=\sup\limits_{r>0}\frac{1}{|B(x,r)|} \int _{B(x,r)}|f(y)|dy
\end{equation}
is bounded on $L^p$ for $1 < p < \infty $, but not on $L^1$,  but using Orlicz spaces, we can investigate
the boundedness of the maximal operator near $p = 1$, see
\cite{315ze, 315zg} and \cite{95zzc} for  more precise
statements.

A natural step in the theory  of functions spaces was to study generalized Orlicz-Morrey spaces
$\mathcal{M}^{\Phi,\vi}(\Rn)$ where the "Morrey-type measuring"of regularity of functions is realized with respect to the Orlicz
norm over balls instead of the Lebesgue one. Such spaces were first introduced and studied by Nakai \cite{Nakai0}. Then another kind of generalized Orlicz-Morrey spaces were introduced by Sawano $et \,al.$ \cite{SawSugTan}. The generalized Orlicz-Morrey spaces used in this paper was introduced in \cite{DerGulSam}. In words of \cite{GulHasSawNak}, our generalized Orlicz-Morrey space is the third kind and the ones in \cite{Nakai0} and \cite{SawSugTan} are the first kind and second kind, respectively. According to the examples in \cite{GalaSawTan}, one can say that the generalized Orlicz-Morrey spaces of the first and second kind are different and that second kind and third kind are different. However, it is not known that relation between first and third kind.

Let $0<\a<n$. The fractional maximal operator $M_{\a}$ and the Riesz potential operator $I_{\a}$ are  defined by
\begin{equation*}
M_{\a} f(x)=\sup_{t>0}|B(x,t)|^{-1+ \frac{\a}{n}}\int _{B(x,t)} |f(y)|dy,
~~~
I_{\a} f(x)=\int _{\Rn} \frac{f(y)}{|x-y|^{n-\a}}dy.
\end{equation*}
If $\a=0$, then $M \equiv M_{0}$ is the Hardy-Littlewood maximal operator defined in \eqref{maxoperator}.

The classical result by Hardy-Littlewood-Sobolev states that the operator $I_{\a} $ is of weak type $(p,np/(n-\a p))$ if $1\le p< n/\a$ and of strong
type $(p,np/(n-\a p))$ if $1< p< n/\a$. Also the operator $M_{\a}$ is of weak type $(p,np/(n-\a p))$ if $1\le p\le n/\a$ and of strong type $(p,np/(n-\a p))$ if $1< p\le n/\a$.

Around the 1970's, the Hardy-Littlewood-Sobolev inequality is extended from Lebesgue spaces to Morrey spaces.
The following theorem was proved by Adams in \cite{Adams}.
\begin{thm} (Adams \cite{Adams}) \label{Adams1}
Let $0<\alpha<n$, $1<p<\frac{n}{\alpha}$, $0<\lambda<n-\alpha p$ and
$\frac{1}{p}-\frac{1}{q}=\frac{\alpha}{n-\lambda}$.
Then for $p>1$, the operator $I_{\alpha}$ is bounded from $\mathcal{M}^{p,\lambda}(\Rn)$
to $\mathcal{M}^{q,\lambda}(\Rn)$ and for $p=1$, $I_{\alpha}$ is bounded from $\mathcal{M}^{1,\lambda}(\Rn)$  to $W\mathcal{M}^{q,\lambda}(\Rn)$.
\end{thm}
Recall that, for $0<\alpha<n$,
\begin{equation*}\label{eq001}
M_{\alpha}f(x)\leq \upsilon_n^{\frac{\alpha}{n}-1}I_{\alpha}(|f|)(x),
\end{equation*}
hence Theorem \ref{Adams1} also implies boundedness of the
fractional maximal operator $M_{\alpha}$.

Guliyev \cite{GulJIA} (see also, \cite{GulDoc,GulBook}) extended the results of Spanne and Adams from Morrey spaces to generalized Morrey spaces (see also \cite{Gun2003}).
Later on, some of these results are obtained in \cite{GulAkbMam,GulShu} under weaker condition.

The boundedness of $M_{\a} $ from Orlicz space $L^{\Phi}(\Rn)$ to the corresponding another
Orlicz space $L^{\Psi}(\Rn)$ was studied in  \cite{95zzc}. There were given  necessary and sufficient conditions
for the boundedness of the operator $M_{\a}$  from $L^{\Phi}(\Rn)$ to $L^{\Psi}(\Rn)$ and also  from $L^{\Phi}(\Rn)$ to the weak Orlicz space $WL^{\Psi}(\Rn)$.

In this paper, we shall give a characterization for weak/strong
Adams type boundedness of the fractional maximal operator on generalized Orlicz-Morrey spaces.

By $A \lesssim B$ we mean that $A \le C B$ with some positive constant $C$
independent of appropriate quantities. If $A \lesssim B$ and $B \lesssim A$, we
write $A\approx B$ and say that $A$ and $B$ are  equivalent.

\section{Preliminaries}
\subsection{On Young Functions and Orlicz Spaces}
We recall the definition of Young functions.

\begin{defn}\label{def2} A function $\Phi : [0,\infty) \rightarrow [0,\infty]$ is called a Young function if $\Phi$ is convex, left-continuous, $\lim\limits_{r\rightarrow +0} \Phi(r) = \Phi(0) = 0$ and $\lim\limits_{r\rightarrow \infty} \Phi(r) = \infty$.
\end{defn}
From the convexity and $\Phi(0) = 0$ it follows that any Young function is increasing.
If there exists $s \in  (0,\infty )$ such that $\Phi(s) = \infty $,
then $\Phi(r) = \infty $ for $r \geq s$.
The set of  Young  functions such that
$0<\Phi(r)<\infty$ for $0<r<\infty$
will be denoted by  $\mathcal{Y}.$
If $\Phi \in  \mathcal{Y}$, then $\Phi$ is absolutely continuous on every closed interval in $[0,\infty )$
and bijective from $[0,\infty )$ to itself.
For a Young function $\Phi$ and  $0 \leq s \leq \infty $, let
$$\Phi^{-1}(s)=\inf\{r\geq 0: \Phi(r)>s\}.$$
If $\Phi \in  \mathcal{Y}$, then $\Phi^{-1}$ is the usual inverse function of $\Phi$.

Note that Young functions satisfy the properties
\begin{equation}\label{cvccv}
\left\{
\begin{array}{ccc}
\Phi(\alpha t)\leq \alpha \Phi(t),
& \text{  if } &0\le\a\le1 \\
\Phi(\a t)\geq \a \Phi(t),&\text{  if }& \a>1
\end{array}
\right.
\text{ and }
\left\{
\begin{array}{ccc}
\Phi^{-1}(\alpha t)\geq \alpha \Phi^{-1}(t),
& \text{  if } &0\le\a\le1 \\
\Phi^{-1}(\a t)\leq \a \Phi^{-1}(t),&\text{  if }& \a>1.
\end{array}
\right.
\end{equation}

\begin{rem}
We can easily see that $\Phi\big(Cr\big)\approx \Phi\big(r\big)$ and $\Phi^{-1}\big(Cr\big)\approx \Phi^{-1}\big(r\big)$ for a positive constant $C$ from \eqref{cvccv}.
\end{rem}

It is well known that
\begin{equation}\label{2.3}
r\leq \Phi^{-1}(r)\widetilde{\Phi}^{-1}(r)\leq 2r \qquad \text{for } r\geq 0,
\end{equation}
where $\widetilde{\Phi}(r)$ is defined by
\begin{equation*}
\widetilde{\Phi}(r)=\left\{
\begin{array}{ccc}
\sup\{rs-\Phi(s): s\in  [0,\infty )\}
& , & r\in  [0,\infty ) \\
\infty &,& r=\infty .
\end{array}
\right.
\end{equation*}

A Young function $\Phi$ is said to satisfy the
 $\Delta_2$-condition, denoted also as   $\Phi \in  \Delta_2$, if
$
\Phi(2r)\le C\Phi(r), \, r>0
$
for some $C>1$. If $\Phi \in  \Delta_2$, then $\Phi \in  \mathcal{Y}$. A Young function $\Phi$ is said to satisfy the
$\nabla_2$-condition, denoted also by  $\Phi \in  \nabla_2$, if
$\Phi(r)\leq \frac{1}{2C}\Phi(Cr), \, r \geq 0$
for some $C>1$.

A Young function $\Phi$ is said to satisfy the
$\Delta^{\prime}$-condition, denoted also as   $\Phi \in  \Delta^{\prime}$, if
$$
\Phi(tr)\le C\Phi(t)\Phi(r),\qquad t,r\geq 0
$$
for some $C>1$.

Note that, each element of $\Delta^{\prime}$-class is also an element of $\Delta_2$-class.

\begin{rem}\label{trvorlmor}
Let $\Phi\in\Delta^{\prime}$, then we have
$$
\Phi(tr)\le C\Phi(t)\Phi(r),\qquad t,r\geq 0.
$$
If we set $\Phi(t)=u$ and $\Phi(r)=v$, we get $\Phi(\Phi^{-1}(u)\Phi^{-1}(v))\leq Cuv\Rightarrow \Phi^{-1}(u)\Phi^{-1}(v)\leq \Phi^{-1}(Cuv)\leq C\Phi^{-1}(uv)$, since $\Phi \in  \mathcal{Y}$ and $\Phi^{-1}$ is concave.
\end{rem}

\begin{defn} (Orlicz Space). For a Young function $\Phi$, the set
\begin{equation*}
L^{\Phi}(\Rn)=\Big\{f\in  L^1_{\rm loc}(\Rn): \int_{\Rn} \Phi(k|f(x)|)dx<\infty
 \text{ for some $k>0$  }\Big\}
\end{equation*}
is called Orlicz space. If $\Phi(r)=r^{p},\, 1\le p<\infty $, then $L^{\Phi}(\Rn)=L^{p}(\Rn)$.
If $\Phi(r)=0, \, 0\le r\le 1$ and $\Phi(r)=\infty ,\, r > 1$, then $L^{\Phi}(\Rn)=L^\infty (\Rn)$. The  space $L^{\Phi}_{\rm loc}(\Rn)$ is defined as the set of all functions $f$ such that  $f\chi_{_B}\in  L^{\Phi}(\Rn)$ for all balls $B \subset \Rn$.
\end{defn}

$L^{\Phi}(\Rn)$ is a Banach space with respect to the norm
$$\|f\|_{L^{\Phi}}=\inf\left\{\lambda>0:\int _{\Rn}\Phi\Big(\frac{|f(x)|}{\lambda}\Big)dx\leq 1\right\}.$$
We note that
\begin{equation}\label{orlpr}
\int _{\Rn}\Phi\Big(\frac{|f(x)|}{\|f\|_{L^{\Phi}}}\Big)dx\leq 1.
\end{equation}

\begin{lem}\label{lemHold}\cite{DerGulSam}
For a Young function $\Phi$ and $B=B(x,r)$, the following inequality is valid
$$\|f\|_{L^{1}(B)} \leq 2 |B| \Phi^{-1}\left(|B|^{-1}\right) \|f\|_{L^{\Phi}(B)},$$
where $\|f\|_{L^{\Phi}(B)}=\|f\chi_{_B}\|_{L^{\Phi}}$.
\end{lem}

By elementary calculations we have the following.
\begin{lem}\label{charorlc}
Let $\Phi$ be a Young function and $B$ be a set in $\mathbb{R}^n$ with finite Lebesgue measure. Then
\begin{equation*}
\|\chi_{_B}\|_{L^{\Phi}} = \frac{1}{\Phi^{-1}\left(|B|^{-1}\right)}.
\end{equation*}
\end{lem}

The following theorem is an analogue of Lebesgue differentiation theorem in Orlicz spaces.

\begin{thm}\label{lebesgue}\cite{HenKlep}
Suppose that $\Phi$ is a Young function and let $f\in L^{\Phi}(\Rn)$ be
nonnegative. Then
\begin{equation*}
\liminf_{r\to 0+}\frac{\|f\chi_{B(x,r)}\|_{L^{\Phi}}}{\|\chi_{B(x,r)}\|_{L^{\Phi}}}\geq f(x), \quad \text{for almost every }x\in\Rn.
\end{equation*}
If we moreover assume that $\Phi\in\Delta^{\prime}$, then
\begin{equation*}
\lim_{r\to 0+}\frac{\|f\chi_{B(x,r)}\|_{L^{\Phi}}}{\|\chi_{B(x,r)}\|_{L^{\Phi}}}=f(x), \quad \text{for almost every }x\in\Rn.
\end{equation*}
\end{thm}

\subsection{Orlicz-Morrey Spaces}

\begin{defn}\label{OrlMor}
For a Young function $\Phi$ and $\lambda\in \R$,
we denote by $\mathcal{M}^{\Phi,\lambda}(\Rn)$ the Orlicz-Morrey space, defined as the space of all
functions $f\in L^{\Phi}_{\rm loc}(\Rn)$ with finite quasinorm
$$
  \left\| f\right\|_{\mathcal{M}^{\Phi,\lambda}}= \sup_{x\in \Rn, \; r>0}  \Phi^{-1}\big(|B(x,r)|^{-\frac{\lambda}{n}}\big) \|f\chi_{B(x,r)}\|_{L^{\Phi}}.
$$
Note that $\mathcal{M}^{\Phi,\lambda}\big|_{\lambda=0}=L^{\Phi}(\Rn)$ and $\mathcal{M}^{\Phi,\lambda}\big|_{\Phi(t)=t^p}=\mathcal{M}^{p,\lambda}(\Rn)$.
\end{defn}

\begin{lem}\label{ominf}
If $\Phi\in\Delta^{\prime}$, then $\mathcal{M}^{\Phi,n}(\Rn)=L^{\i}(\Rn)$.
\end{lem}

\begin{proof}
Let $f\in L^{\i}(\Rn)$, then
$$
\Phi^{-1}\big(|B(x,r)|^{-1}\big) \|f\chi_{B(x,r)}\|_{L^{\Phi}}\leq \|f\|_{L^{\i}}\Phi^{-1}\big(|B(x,r)|^{-1}\big)\|\chi_{B(x,r)}\|_{L^{\Phi}}\leq \|f\|_{L^{\i}},
$$
which implies $\left\| f\right\|_{\mathcal{M}^{\Phi,n}}\leq \|f\|_{L^{\i}}$.

Now let $f\in \mathcal{M}^{\Phi,n}(\Rn)$. Theorem \ref{lebesgue} implies that $|f(x)|\leq \left\| f\right\|_{\mathcal{M}^{\Phi,n}}$ for almost every $x\in\Rn$, which means that $\left\| f\right\|_{\mathcal{M}^{\Phi,n}}\geq \|f\|_{L^{\i}}$.
\end{proof}
In the following we denote by $\Theta$ the set of all functions equivalent to 0 on $\Rn$.

\begin{lem}\label{ntom}
Let $\Phi$ be a Young function.
If $\lambda<0$ or $\lambda>n$ and $\Phi\in\Delta^{\prime}$, then $\mathcal{M}^{\Phi,\lambda}(\Rn)=\Theta$.
\end{lem}

\begin{proof}
First let $\lambda<0$ and $f\in \mathcal{M}^{\Phi,\lambda}(\Rn)$. For all $x\in\Rn$ and $r>0$ we have
$$
\|f\chi_{B(x,r)}\|_{L^{\Phi}}\leq \frac{\left\| f\right\|_{\mathcal{M}^{\Phi,\lambda}}}{\Phi^{-1}\big(|B(x,r)|^{-\frac{\lambda}{n}}\big)},
$$
which implies that $\|f\|_{L^{\Phi}}=\lim_{r\to\i}\|f\chi_{B(x,r)}\|_{L^{\Phi}}=0\Rightarrow f(x)=0,$ for almost every $x\in\Rn$, since $\lim\limits_{r\to\i}\Phi^{-1}(r)=\i$.

Now let $\lambda>n$ and $f\in \mathcal{M}^{\Phi,\lambda}(\Rn)$. For all $x\in\Rn$ and $r>0$ we have from Remark \ref{trvorlmor} and Theorem \ref{lebesgue}
$$
\frac{\|f\chi_{B(x,r)}\|_{L^{\Phi}}}{\|\chi_{B(x,r)}\|_{L^{\Phi}}}\leq \frac{\Phi^{-1}\big(|B(x,r)|^{-1}\big)}{\Phi^{-1}\big(|B(x,r)|^{-1}|B(x,r)|^{1-\frac{\lambda}{n}}\big)}\left\| f\right\|_{\mathcal{M}^{\Phi,\lambda}}\leq \frac{\left\| f\right\|_{\mathcal{M}^{\Phi,\lambda}}}{\Phi^{-1}\big(|B(x,r)|^{1-\frac{\lambda}{n}}\big)},
$$
which implies that $|f(x)|=0,$ for almost every $x\in\Rn$.
\end{proof}

\begin{rem}
In the case $\Phi(t)=t^p$ for $1\le p<\i$ from Lemmas \ref{ominf} and \ref{ntom} we get the following well known results: $\mathcal{M}^{p,n}(\Rn)=L^{\i}(\Rn)$ and $\mathcal{M}^{p,\lambda}(\Rn)= \Theta$ for $\lambda < 0$ or $\lambda > n$.
\end{rem}

\subsection{Generalized Orlicz-Morrey Spaces}

Various versions of generalized Orlicz-Morrey spaces were introduced in \cite{Nakai0}, \cite{SawSugTan} and \cite{DerGulSam}.
We used the definition of \cite{DerGulSam} which runs as follows.

\begin{defn}\label{genOrlMor}
Let $\varphi(x,r)$ be a positive measurable function on $\Rn \times (0,\infty)$ and $\Phi$ any Young function.
We denote by $\mathcal{M}^{\Phi,\varphi}(\Rn)$ the generalized Orlicz-Morrey space, the space of all
functions $f\in L^{\Phi}_{\rm loc}(\Rn)$ for which
$$
\|f\|_{\mathcal{M}^{\Phi,\varphi}} = \sup\limits_{x\in\Rn, r>0}
\varphi(x,r)^{-1} \Phi^{-1}(|B(x,r)|^{-1}) \|f\|_{L^{\Phi}(B(x,r))}<\infty.
$$
\end{defn}

In the case $\varphi(x,r)=\frac{\Phi^{-1}\big(|B(x,r)|^{-1}\big)}{\Phi^{-1}\big(|B(x,r)|^{-\lambda/n}\big)}$, we get the Orlicz-Morrey space $\mathcal{M}^{\Phi,\lambda}(\Rn)$
from generalized Orlicz-Morrey space $\mathcal{M}^{\Phi,\vi}(\Rn)$ .

\begin{lem}\label{Lemma1Orl} Let $\Phi$ be a Young function and $ \varphi $ be a positive measurable function on $\Rn\times (0,\infty)$.
\begin{itemize}
\item[(i)] If
\begin{align}\label{L11Orl}
\sup_{ t<r<\infty }\frac{\Phi^{-1}(|B(x,r)|^{-1})}{\varphi(x,r)}=\infty\quad\textrm{ for some } t>0~\textrm{ and for all } x\in\Rn,
\end{align}
then $ \mathcal{M}^{\Phi,\varphi}(\Rn)=\Theta$.
\item[(ii)] If $\Phi\in\Delta^{\prime}$ and
\begin{align}\label{L12Orl}
\sup_{ 0<r<\tau} \varphi(x,r)^{-1}=\infty\quad\textrm{ for some } \tau>0~\textrm{ and for all } x\in\Rn,
\end{align}
then $ \mathcal{M}^{\Phi,\varphi}(\Rn)=\Theta $.
\end{itemize}
\end{lem}
\begin{proof} (i) Let \eqref{L11Orl} be satisfied and $ f $ be not equivalent to zero. Then $ \sup_{x\in\Rn}\left\| f\right\|_{L^\Phi(B(x,t))}>0 $, hence
\begin{align*}
\left\| f\right\|_{\mathcal{M}^{\Phi,\varphi}}&\geq\sup_{x\in\Rn}\sup_{ t<r<\infty }\varphi(x,r)^{-1} \Phi^{-1}(|B(x,r)|^{-1}) \left\| f\right\|_{L^\Phi(B(x,r))}
\\
&\geq \sup_{x\in\Rn}\left\| f\right\|_{L^\Phi(B(x,t))}\sup_{ t<r<\infty }\varphi(x,r)^{-1}\Phi^{-1}(|B(x,r)|^{-1}).
\end{align*}
Therefore $ \left\| f\right\|_{\mathcal{M}^{\Phi,\varphi}}=\infty$.\\
(ii) Let $ f\in \mathcal{M}^{\Phi,\varphi}(\Rn) $ and \eqref{L12Orl} be satisfied. Then there are two possibilities:

\item Case 1: $\sup_{ 0<r<t }\varphi(x,r)^{-1}=\infty$ for all $t>0$.

\item Case 2: $\sup_{ 0<r<s }\varphi(x,r)^{-1}<\infty$ for some $s\in(0,\tau)$.

For Case 1, by Theorem \ref{lebesgue}, for almost all $x\in\Rn$,
\begin{align}\label{LebDifOrl}
\lim_{r\to 0+}\frac{\|f\chi_{B(x,r)}\|_{L^{\Phi}}}{\|\chi_{B(x,r)}\|_{L^{\Phi}}}=|f(x)|.
\end{align}
We claim that $ f(x)=0 $ for all those $ x$. Indeed, fix $x$ and assume $ |f(x)|>0$. Then by Lemma \ref{charorlc} and \eqref{LebDifOrl} there exists $ t_0>0 $ such that
\begin{align*}
\Phi^{-1}\big(|B(x,r)|^{-1}\big)\left\| f\right\|_{L^\Phi(B(x,r))}\geq \frac{|f(x)|}{2}
\end{align*}
for all $ 0<r\leq t_0$. Consequently,
\begin{align*}
\left\| f\right\|_{\mathcal{M}^{\Phi,\varphi}} & \geq \sup_{0<r<t_0}\varphi(x,r)^{-1} \Phi^{-1}\big(|B(x,r)|^{-1}\big)
 \left\| f\right\|_{L^\Phi(B(x,r))}
 \\
 & \geq \frac{|f(x)|}{2}\sup_{0<r<t_0}\varphi(x,r)^{-1}.
\end{align*}
Hence $ \left\| f\right\|_{\mathcal{M}^{\Phi,\varphi}}=\infty$, so $ f\notin \mathcal{M}^{\Phi,\varphi}(\Rn) $ and we have arrived at a contradiction.

Note that Case 2 implies that $\sup_{s<r<\tau}\varphi(x,r)^{-1}=\infty$, hence
\begin{align*}
\sup_{s<r<\infty}\varphi(x,r)^{-1}\Phi^{-1}(|B(x,r)|^{-1})&\geq\sup_{s<r<\tau}\varphi(x,r)^{-1}\Phi^{-1}(|B(x,r)|^{-1})
\\
&\geq \Phi^{-1}(|B(x,\tau)|^{-1})\sup_{s<r<\tau}\varphi(x,r)^{-1}=\infty,
\end{align*}
which is the case in (i).
\end{proof}

\begin{rem}\label{nontrivorlmor}
Let $\Phi$ be a Young function. We denote by $\Omega_{\Phi}$ the sets of all positive measurable functions $\varphi$ on $\Rn\times (0,\infty)$ such that for all $t>0$,
$$
\sup_{x\in\Rn} \Big\|\frac{\Phi^{-1}(|B(x,r)|^{-1})}{\varphi(x,r)} \Big\|_{L^\infty(t,\infty)}<\infty,
$$
and
$$
\sup_{x\in\Rn}\Big\|\varphi(x,r)^{-1} \Big\|_{L^\infty(0, t)}<\infty,
$$
respectively. In what follows, keeping in mind Lemma \ref{Lemma1Orl}, we always assume that $\varphi\in\Omega_{\Phi}$ and $\Phi\in\Delta^{\prime}$.
\end{rem}

The following theorem and lemma play a key role in our main results.
\begin{thm}\label{thm4.4.}\cite{DerGulSam}
Let $\Phi$ be a Young function, the functions $\varphi \in \Omega_{\Phi}$ and $\Phi \in \Delta^{\prime}$ satisfy the condition
\begin{equation}\label{bounmax}
\sup_{r<t<\infty} \Phi^{-1}\big(|B(x,t)|^{-1}\big) \es_{t<s<\infty}\frac{\varphi(x,s)}{\Phi^{-1}\big(|B(x,s)|^{-1}\big)} \le C \, \varphi(x,r),
\end{equation}
where $C$ does not depend on $x$ and $r$. Then the maximal operator $M$ is bounded from $\mathcal{M}^{\Phi,\varphi}(\Rn)$ to $\mathcal{M}^{\Phi,\varphi}(\Rn)$ for $\Phi \in \nabla_2$.
\end{thm}

A function $\varphi:(0,\infty) \to (0,\infty)$ is said to be almost increasing (resp.
almost decreasing) if there exists a constant $C > 0$ such that
$$
\varphi(r)\leq C \varphi(s)\qquad (\text{resp. }\varphi(r)\geq C \varphi(s))\quad \text{for  } r\leq s.
$$
For a Young function $\Phi$, we denote by ${\mathcal{G}}_{\Phi}$ the set of all almost decreasing functions $\varphi:(0,\infty) \to (0,\infty)$
such that $t\in(0,\infty) \mapsto \frac{\varphi(t)}{\Phi^{-1}(t^{-n})}$ is almost increasing.

\begin{lem}\label{charOrlMor}\cite{GulDerPot}
Let $B_0:=B(x_0,r_0)$. If $\varphi\in{\mathcal{G}}_{\Phi}$, then there exist $C>0$ such that
$$
\frac{1}{\varphi(r_0)}\leq \|\chi_{B_0}\|_{\mathcal{M}^{\Phi,\varphi}}\leq \frac{C}{\varphi(r_0)}.
$$
\end{lem}

\section{Adams type results for $M_{\alpha}$ in $\mathcal{M}^{\Phi,\varphi}$}
For proving our main results, we need the following estimate.
\begin{lem}\label{estFrMax}
If $B_0:=B(x_0,r_0)$, then $r_0^{\alpha}\leq C M_{\a} \chi_{B_0}(x)$ for every $x\in B_0$.
\end{lem}
\begin{proof}
It is well known that
\begin{equation}\label{poicomcumax}
\mathrm{M}_{\a}f(x)\leq 2^{n-\a}M_{\a}f(x),
\end{equation}
where $\mathrm{M}_{\a}(f)(x)=\sup\limits_{B\ni x}|B|^{-1+ \frac{\a}{n}}\int _{B} |f(y)|dy$.

Now let $x\in B_0$. By using \eqref{poicomcumax}, we get
\begin{align*}
M_{\a} \chi_{B_0}(x)&\geq C \mathrm{M}_{\a} \chi_{B_0}(x)\geq C\sup\limits_{B\ni x}|B|^{-1+ \frac{\a}{n}}|B \cap B_0|
\\
& \geq C |B_0|^{-1+ \frac{\a}{n}}|B_0 \cap B_0|= C r_0^{\a}.
\end{align*}
\end{proof}

\begin{thm}\label{AdGulFrMaxOrlMor}
Let $\Phi\in \Delta^{\prime}\cap \nabla_2$ and $0< \a<n$. Let $\varphi \in \Omega_{\Phi}$ satisfy the conditions \eqref{bounmax} and
\begin{equation}\label{eq3.6.Vfrmax}
r^{\alpha}\varphi(x,r) + \sup_{r<t<\infty} t^{\alpha}\, \varphi(x,t) \le C
\varphi(x,r)^{\beta},
\end{equation}
for some $\beta\in(0,1)$ and for every $x\in\Rn$ and $r>0$.
Define $\eta(x,r)\equiv\varphi(x,r)^{\beta}$ and $\Psi(r)\equiv\Phi(r^{1/\beta})$. Then the operator $M_\a$ is bounded from $\mathcal{M}^{\Phi,\varphi}(\Rn)$ to $\mathcal{M}^{\Psi,\eta}(\Rn)$.
\end{thm}
\begin{proof}
For arbitrary ball $B=B(x,r)$ we represent $f$ as
\begin{equation*}
f=f_1+f_2, \ \quad f_1(y)=f(y)\chi _{2B}(y),\quad
 f_2(y)=f(y)\chi_{\dual {(2B)}}(y), \ \quad r>0,
\end{equation*}
and have
$$
M_\a f(x)=M_\a f_1(x)+M_\a f_2(x).
$$

Let $y$ be an arbitrary point in $B$. If $B(y,t)\cap {\dual}(B(x,2r))\neq\emptyset,$ then $t>r$. Indeed, if $z\in B(y,t)\cap  {\dual} (B(x,2r)),$
then $t > |y-z| \geq |x-z|-|x-y|>2r-r=r$.

On the other hand, $B(y,t)\cap {\dual} (B(x,2r))\subset B(x,2t)$. Indeed, if  $z\in B(y,t)\cap {\dual} (B(x,2r))$, then
we get $|x-z|\leq |y-z|+|x-y|<t+r<2t$.

Hence
\begin{equation*}
\begin{split}
M_{\a} f_2(y) & = \sup_{t>0}\frac{1}{|B(y,t)|^{1-\frac{\a}{n}}} \int_{B(y,t)\cap {{\dual}(B(x,2r))}}|f(z)|d z
\\
& \le 2^{n-\a} \, \sup_{t>r}\frac{1}{|B(x,2t)|^{1-\frac{\a}{n}}} \int_{B(x,2t)}|f(z)|d z
\\
&= 2^{n-\a} \, \sup_{t>2r} \frac{1}{|B(x,t)|^{1-\frac{\a}{n}}} \int_{B(x,t)}|f(z)| d z
\\
&\le C_2\sup_{r<t<\infty} \Phi^{-1}(|B(x,t)|^{-1})t^{\a}\|f\|_{L^{\Phi}(B(x,t))}.
\end{split}
\end{equation*}
Consequently from Hedberg's trick and the last inequality, we have
\begin{equation*}
\begin{split}
M_\a f(y) &
\lesssim
r^\a Mf(y)+\sup_{r<t<\infty} \Phi^{-1}(|B(x,t)|^{-1})t^{\a}\|f\|_{L^{\Phi}(B(x,t))}
\\
&\lesssim r^\a Mf(y)+\|f\|_{\mathcal{M}^{\Phi,\varphi}}\sup_{r<t<\infty} t^\a\varphi(x,t).
\end{split}
\end{equation*}

Thus, using the technique in \cite[p. 6492]{SawSugTan2}, by \eqref{eq3.6.Vfrmax} we obtain
\begin{align*}
|M_{\a} f(y)| & \lesssim  \min \{ \varphi(x,r)^{\beta-1} Mf(y), \varphi(x,r)^{\beta} \|f\|_{\mathcal{M}^{\Phi,\varphi}}\}
\\
& \lesssim \sup\limits_{s>0} \min \{ s^{\beta-1} Mf(y), s^{\beta} \|f\|_{\mathcal{M}^{\Phi,\varphi}}\}
\\
& = (Mf(y))^{\beta} \, \|f\|_{\mathcal{M}^{\Phi,\varphi}}^{1-\beta},
\end{align*}
where we have used that the supremum is achieved when the minimum parts are balanced.
Hence for every $y\in B$ we have
\begin{equation}\label{poiwseestfrmax}
M_{\a} f(y) \lesssim  (Mf(y))^{\beta} \, \|f\|_{\mathcal{M}^{\Phi,\varphi}}^{1-\beta}.
\end{equation}

By using the inequality \eqref{poiwseestfrmax} we have
$$
\|M_{\a} f\|_{L^{\Psi}(B)}\lesssim \|(Mf)^{\beta}\|_{L^{\Psi}(B)}\, \|f\|_{\mathcal{M}^{\Phi,\varphi}}^{1-\beta}.
$$

Note that from \eqref{orlpr} we get
$$
\int_B \Psi\left(\frac{(Mf(x))^{\beta}}{\|Mf\|_{L^{\Phi}(B)}^{\beta}}\right)dx=\int_B \Phi\left(\frac{Mf(x)}{\|Mf\|_{L^{\Phi}(B)}}\right)dx\leq 1.
$$
Thus $\|(Mf)^{\beta}\|_{L^{\Psi}(B)}\leq \|Mf\|_{L^{\Phi}(B)}^{\beta}$. Consequently by using the boundedness of the maximal operator, we get
\begin{align*}
& \eta(x,r)^{-1}\Psi^{-1}(|B|^{-1})\|M_{\a} f\|_{L^{\Psi}(B)}\lesssim \eta(x,r)^{-1}\Psi^{-1}(|B|^{-1})\|Mf\|_{L^{\Phi}(B)}^{\beta}\, \|f\|_{\mathcal{M}^{\Phi,\varphi}}^{1-\beta}\\
&=\left(\varphi(x,r)^{-1}\Phi^{-1}(|B|^{-1})\|Mf\|_{L^{\Phi}(B)}\right)^{\beta}\, \|f\|_{\mathcal{M}^{\Phi,\varphi}}^{1-\beta} \lesssim \|f\|_{\mathcal{M}^{\Phi,\varphi}}.
\end{align*}
By taking the supremum of all $B$, we get the desired result.
\end{proof}
\begin{rem}
Note that, for $\eta(x,r)\equiv\varphi(x,r)^{\beta}$ and $\Psi(r)\equiv\Phi(r^{1/\beta})$, $\varphi \in \Omega_{\Phi}$ implies  that $\eta \in \Omega_{\Psi}$.
Also, $\Phi\in \Delta^{\prime}$ implies $\Psi\in \Delta^{\prime}$.
\end{rem}

The following theorem is one of our main results.
\begin{thm} (Adams type result) \label{AdGulFrMaxOrlMorNec}

Let $0<\a<n$, $\Phi\in \Delta^{\prime}$, $\varphi \in \Omega_{\Phi}$, $\beta\in(0,1)$, $\eta(t)\equiv\varphi(t)^{\beta}$ and $\Psi(t)\equiv\Phi(t^{1/\beta})$.

$1.~$ If $\Phi\in\nabla_2$ and $\varphi(t)$
satisfies \eqref{bounmax}, then the condition
\begin{equation}\label{eq3.6.V}
t^{\alpha}\varphi(t) + \sup_{t<r<\infty} r^{\alpha}\, \varphi(r)  \le C
\varphi(t)^{\beta},
\end{equation}
for all $t>0$, where $C>0$ does not depend on $t$, is sufficient for the boundedness of $M_{\a}$ from $\mathcal{M}^{\Phi,\varphi}(\Rn)$ to $\mathcal{M}^{\Psi,\eta}(\Rn)$.

$2.~$ If $\varphi\in{\mathcal{G}}_{\Phi}$, then the condition
\begin{equation}\label{condAdams}
t^{\alpha}\varphi(t)\le C \varphi(t)^{\beta},
\end{equation}
for all $t>0$, where $C>0$ does not depend on $t$, is necessary for the boundedness of $M_{\a}$ from $\mathcal{M}^{\Phi,\varphi}(\Rn)$ to $\mathcal{M}^{\Psi,\eta}(\Rn)$.

$3.~$ Let $\Phi\in\nabla_2$ and $\varphi\in{\mathcal{G}}_{\Phi}$. Then, the condition \eqref{condAdams}
is necessary and sufficient for the boundedness of $M_{\a}$ from $\mathcal{M}^{\Phi,\varphi}(\Rn)$ to $\mathcal{M}^{\Psi,\eta}(\Rn)$.
\end{thm}

\begin{proof}
The first part of the theorem is a corollary of Theorem \ref{AdGulFrMaxOrlMor}.

We shall now prove the necessary part. Let $B_0=B(x_0,t_0)$ and $x\in B_0$. By Lemma \ref{estFrMax} we have $t_0^{\alpha}\leq C M_{\a} \chi_{B_0}(x)$. Therefore, by by Lemmas \ref{charorlc} and \ref{charOrlMor}
\begin{align*}
t_0^{\alpha}&\leq C\Psi^{-1}(|B_0|^{-1})\|M_{\a} \chi_{B_0}\|_{L^{\Psi}(B_0)} \leq C\eta(t_0)\|M_{\a} \chi_{B_0}\|_{\mathcal{M}^{\Psi,\eta}} \\
&\leq C\eta(t_0)\|\chi_{B_0}\|_{\mathcal{M}^{\Phi,\varphi}}\leq C\frac{\eta(t_0)}{\varphi(t_0)}\leq C \varphi(t_0)^{\beta-1}.
\end{align*}
Since this is true for every $t_0>0$, we are done.
The third statement of the theorem follows from the first and second parts of the theorem.
\end{proof}

If we take $\Phi(t)=t^{p},\,p\in[1,\infty)$ and $\beta=\frac{p}{q}$ with $p < q<\infty$ at Theorem \ref{AdGulFrMaxOrlMorNec}
we get the following result for the generalized Morrey spaces which also can be seen as a special case of \cite[Theorem 1]{HakNakSaw}.
\begin{cor}\label{gmcorad}

Let $0< \alpha<n$, $1<p < q<\infty$ and $\varphi \in \Omega_{p} \equiv \Omega_{t^{p}}$.

$1.~$ If $\varphi(t)$ satisfies
\begin{equation}\label{eq4.6.GSMax}
\sup_{r<t<\infty} \frac{\es_{t<s<\i}\varphi(s)s^{\frac{n}{p}}}{t^{\frac{n}{p}}}\le C\varphi(r),
\end{equation}
then the condition
\begin{equation}\label{eq3.6.VGM}
t^{\alpha}\varphi(t) + \sup_{t<r<\infty} r^{\alpha}\, \varphi(r)  \le C
\varphi(t)^{\frac{p}{q}},
\end{equation}
for all $t>0$, where $C>0$ does not depend on $t$, is sufficient for the boundedness of $M_{\a}$ from $\mathcal{M}^{p,\varphi}(\Rn)$ to $\mathcal{M}^{q,\varphi^{\frac{p}{q}}}(\Rn)$.

$2.~$ If $\varphi\in{\mathcal{G}}_{p}\equiv \mathcal{G}_{t^{p}}$, then the condition
\begin{equation}\label{condad}
t^{\alpha}\varphi(t)\le C \varphi(t)^{\frac{p}{q}},
\end{equation}
for all $t>0$, where $C>0$ does not depend on $t$, is necessary for the boundedness of $M_{\a}$ from $\mathcal{M}^{p,\varphi}(\Rn)$ to $\mathcal{M}^{q,\varphi^{\frac{p}{q}}}(\Rn)$.

$3.~$ If $\varphi\in{\mathcal{G}}_{p}$, then the condition \eqref{condad}
is necessary and sufficient for the boundedness of $M_{\a}$ from $\mathcal{M}^{p,\varphi}(\Rn)$ to $\mathcal{M}^{q,\varphi^{\frac{p}{q}}}(\Rn)$.
\end{cor}

If we take $\varphi(t)=\frac{\Phi^{-1}\big(t^{-n}\big)}{\Phi^{-1}\big(t^{-\lambda}\big)}$, $0 \le \lambda \le n$, $\Psi(t)\equiv\Phi(t^{1/\beta})$, $\beta\in(0,1)$,
$$
\eta(t)\equiv\varphi(t)^{\beta}=\Big(\frac{\Phi^{-1}\big(t^{-n}\big)}{\Phi^{-1}\big(t^{-\lambda}\big)}\Big)^{\beta}
=\frac{\Psi^{-1}\big(t^{-n}\big)}{\Psi^{-1}\big(t^{-\lambda}\big)},
$$
at Theorem \ref{AdGulFrMaxOrlMorNec} we get the following new result for Orlicz-Morrey spaces.
\begin{cor}\label{MarchN6}
Let $\Phi\in \Delta^{\prime}\cap \nabla_2$, $\Psi(t)\equiv\Phi(t^{1/\beta})$ and $\beta\in(0,1)$. If
\begin{equation}\label{intcondAdamsOrlMorNw}
\sup_{t<r<\infty} r^{\alpha}\, \frac{\Phi^{-1}\big(r^{-n}\big)}{\Phi^{-1}\big(r^{-\lambda}\big)}  \le C t^{\a}\frac{\Phi^{-1}\big(t^{-n}\big)}{\Phi^{-1}\big(t^{-\lambda}\big)},
\end{equation}
for all $t>0$, where $C>0$ does not depend on $t$, then the condition
\begin{equation}\label{NScondAdamsOrlMorNw}
t^{\alpha}  \le C \left[\frac{\Phi^{-1}\big(t^{-n}\big)}{\Phi^{-1}\big(t^{-\lambda}\big)}\right]^{\beta-1}
\end{equation}
for all $t>0$, where $C>0$ does not depend on $t$, is necessary and sufficient for the boundedness of $M_{\a}$ from $\mathcal{M}^{\Phi,\lambda}(\Rn)$ to $\mathcal{M}^{\Psi,\lambda}(\Rn)$.
\end{cor}

\begin{rem}\label{fat82}
If we take $\Phi(t)=t^{p}$, $\beta=\frac{p}{q}$ with $p < q$  at Corollary \ref{MarchN6},
then condition \eqref{intcondAdamsOrlMorNw} is equivalent to $0\le \lambda < n-\a p$ and condition \eqref{NScondAdamsOrlMorNw} is equivalent to $\frac{1}{p}-\frac{1}{q}=\frac{\alpha}{n-\lambda}$. Therefore, we get the following Adams result for Morrey spaces (see Theorem \ref{Adams1}).
\end{rem}

\begin{cor}
Let $0 < \alpha<n$, $1<p< q<\infty$ and $0\le\lambda< n-\alpha p$. Then
$M_{\alpha}$ is bounded from $\mathcal{M}^{p,\lambda}(\Rn)$
to $\mathcal{M}^{q,\lambda}(\Rn)$  if and only if $\frac{1}{p}-\frac{1}{q}=\frac{\alpha}{n-\lambda}$.
\end{cor}

To compare, we formulate the following theorem proved in \cite{GulDerPot} and remark below.

\begin{thm} \label{AdGulRszOrlMorNec}

Let $0<\a<n$, $\Phi\in \Delta^{\prime}$, $\varphi \in \Omega_{\Phi}$, $\beta\in(0,1)$, $\eta(t)\equiv\varphi(t)^{\beta}$ and $\Psi(t)\equiv\Phi(t^{1/\beta})$.

$1.~$ If $\Phi\in\nabla_2$ and $\varphi(t)$
satisfies \eqref{bounmax}, then the condition
\begin{equation*}
t^{\alpha}\varphi(t) + \int_{t}^{\infty} r^{\alpha}\, \varphi(r) \frac{dr}{r} \le C
\varphi(t)^{\beta},
\end{equation*}
for all $t>0$, where $C>0$ does not depend on $t$, is sufficient for the boundedness of $I_{\a}$ from $\mathcal{M}^{\Phi,\varphi}(\Rn)$ to $\mathcal{M}^{\Psi,\eta}(\Rn)$.

$2.~$ If $\varphi\in{\mathcal{G}}_{\Phi}$, then the condition \eqref{condAdams} is necessary for the boundedness of $I_{\a}$ from $\mathcal{M}^{\Phi,\varphi}(\Rn)$ to $\mathcal{M}^{\Psi,\eta}(\Rn)$.

$3.~$ Let $\Phi\in\nabla_2$. If $\varphi\in{\mathcal{G}}_{\Phi}$ satisfies the regularity condition
\begin{equation}\label{intcondAdamsx}
\int_{t}^{\infty} r^{\alpha}\, \varphi(r) \frac{dr}{r} \le C t^{\a}\varphi(t),
\end{equation}
for all $t>0$, where $C>0$ does not depend on $t$, then the condition \eqref{condAdams}
is necessary and sufficient for the boundedness of $I_{\a}$ from $\mathcal{M}^{\Phi,\varphi}(\Rn)$ to $\mathcal{M}^{\Psi,\eta}(\Rn)$.
\end{thm}

\begin{rem}\label{supzyg} Although fractional maximal function is pointwise dominated by the Riesz potential, and consequently, the results for the former could be derived from the results for the latter, we consider them separately, because  we are able to study the fractional maximal operator under weaker assumptions than it derived from the results for the potential operator. More precisely, for $\varphi\in{\mathcal{G}}_{\Phi}$, we don't need to regularity condition \eqref{intcondAdamsx} for the boundedness of fractional maximal operator.
\end{rem}

\section{Weak type results}

\begin{defn} Let $\Phi$ be a Young function. The weak Orlicz space is defined as
$$
WL^{\Phi}(\mathbb{R}^{n}):=\{f\in  L_{\mathrm{loc}}^1(\mathbb{R}^{n}):\Vert f\Vert_{WL^{\Phi}}<\infty \},
$$
where
$$
\Vert f\Vert_{WL^{\Phi}}=\inf\left\{\lambda>0\ :\ \sup_{t>0}\Phi(\frac{t}{\lambda})d_{f}(t)\ \leq 1\right\},
$$
and $d_{f}(t)=|\{x\in\Rn: |f(x)|>t\}|$.
\end{defn}

\begin{lem}
If $0<\Vert f\Vert_{WL^{\Phi}}<\infty$, then
\begin{equation}\label{worlpr}
\sup_{t>0}\Phi(\frac{t}{\Vert f\Vert_{WL^{\Phi}}})d_{f}(t)\ \leq 1.
\end{equation}
\end{lem}

\begin{proof}
By the definition of $\Vert \cdot\Vert_{WL^{\Phi}}$, for all $\lambda>\Vert f\Vert_{WL^{\Phi}}$ we have $\Phi(\frac{t}{\lambda})d_{f}(t)\ \leq 1,~\forall t>0$. Now as $\lambda$ decreases to $\Vert f\Vert_{WL^{\Phi}}$, the quotient $\frac{t}{\lambda}$ increases to $\frac{t}{\Vert f\Vert_{WL^{\Phi}}}$. By the left-continuity of $\Phi$, we have $\Phi\left(\frac{t}{\lambda}\right)\uparrow\Phi\left(\frac{t}{\Vert f\Vert_{WL^{\Phi}}}\right)$. Therefore we get the desired result.
\end{proof}

By elementary calculations we have the following.
\begin{lem}\label{charorlcw}
Let $\Phi$ be a Young function and $B$ a set in $\mathbb{R}^n$ with finite Lebesgue measure. Then
\begin{equation*}
\|\chi_{_B}\|_{WL^{\Phi}} = \frac{1}{\Phi^{-1}\left(|B|^{-1}\right)}.
\end{equation*}
\end{lem}

\begin{defn}\label{WOrlMor}
For a Young function $\Phi$ and $\lambda\in \R$,
we denote by $W\mathcal{M}^{\Phi,\lambda}(\Rn)$ the weak Orlicz-Morrey space, defined as the space of all
functions $f\in WL^{\Phi}_{\rm loc}(\Rn)$ with finite quasinorm
$$
  \left\| f\right\|_{W\mathcal{M}^{\Phi,\lambda}}= \sup_{x\in \Rn, \; r>0}  \Phi^{-1}\big(|B(x,r)|^{-\frac{\lambda}{n}}\big) \|f\chi_{B(x,r)}\|_{WL^{\Phi}}.
$$
\end{defn}

\begin{defn}\label{genOrlMorW}
Let $\varphi(x,r)$ be a positive measurable function on $\Rn \times (0,\infty)$ and $\Phi$ any Young function.
By $W\mathcal{M}^{\Phi,\varphi}(\Rn)$ we denote the weak generalized Orlicz-Morrey space of all functions $f\in WL^{\Phi}_{\rm loc}(\Rn)$ for which
$$
\|f\|_{W\mathcal{M}^{\Phi,\varphi}} = \sup\limits_{x\in\Rn, r>0} \varphi(x,r)^{-1} \Phi^{-1}(|B(x,r)|^{-1}) \|f\|_{WL^{\Phi}(B(x,r))} < \infty.
$$
\end{defn}

\begin{lem}\label{charWOrlMor}
Let $B_0:=B(x_0,r_0)$. If $\varphi\in{\mathcal{G}}_{\Phi}$, then
$
\|\chi_{B_0}\|_{W\mathcal{M}^{\Phi,\varphi}}\approx \varphi(r_0)^{-1}.
$
\end{lem}
\begin{proof}
The proof is similar to the proof of Lemma \ref{charOrlMor} thanks to the Lemma \ref{charorlcw}.
\end{proof}

\begin{thm}\label{thm4.4.w}\cite{DerGulSam}
Let $\Phi$ be a Young function, the functions $\varphi \in \Omega_{\Phi}$ and $\Phi\in\Delta^{\prime}$ satisfy the condition
\eqref{bounmax}, then the maximal operator $M$ is bounded from $\mathcal{M}^{\Phi,\varphi}(\Rn)$ to $W\mathcal{M}^{\Phi,\varphi}(\Rn)$.
\end{thm}

\begin{thm}\label{AdGulFrMaxOrlMorW}
Let $\Phi\in \Delta^{\prime}$ and $0< \a<n$. Let $\varphi \in \Omega_{\Phi}$ satisfy the conditions \eqref{bounmax} and
\eqref{eq3.6.Vfrmax}.
Define $\eta(x,t)\equiv\varphi(x,t)^{\beta}$ and $\Psi(t)\equiv\Phi(t^{1/\beta})$ for $\beta\in(0,1)$. Then the operator $M_\a$ is bounded from $\mathcal{M}^{\Phi,\varphi}(\Rn)$ to $W\mathcal{M}^{\Psi,\eta}(\Rn)$.
\end{thm}
\begin{proof}
By using the inequality \eqref{poiwseestfrmax} we have
$$
\|M_{\a} f\|_{WL^{\Psi}(B)}\lesssim \|(Mf)^{\beta}\|_{WL^{\Psi}(B)}\, \|f\|_{\mathcal{M}^{\Phi,\varphi}}^{1-\beta},
$$
where $B=B(x,r)$.

Note that from \eqref{worlpr} we get
$$
\sup_{t>0}\Psi\left(\frac{t^\beta}{\|Mf\|_{WL^{\Phi}(B)}^{\beta}}\right)d_{(Mf)^{\beta}}(t^\beta)=\sup_{t>0} \Phi\left(\frac{t}{\|Mf\|_{WL^{\Phi}(B)}}\right)d_{Mf}(t)\leq 1.
$$
Thus $\|(Mf)^{\beta}\|_{WL^{\Psi}(B)}\leq \|Mf\|_{WL^{\Phi}(B)}^{\beta}$. Consequently by using the weak boundedness of the maximal operator, we get
\begin{align*}
&\eta(x,r)^{-1}\Psi^{-1}(|B|^{-1})\|M_{\a} f\|_{WL^{\Psi}(B)}\lesssim \eta(x,r)^{-1}\Psi^{-1}(|B|^{-1})\|Mf\|_{WL^{\Phi}(B)}^{\beta}\, \|f\|_{\mathcal{M}^{\Phi,\varphi}}^{1-\beta}\\
&=\left(\varphi(x,r)^{-1}\Phi^{-1}(|B|^{-1})\|Mf\|_{WL^{\Phi}(B)}\right)^{\beta}\, \|f\|_{\mathcal{M}^{\Phi,\varphi}}^{1-\beta} \lesssim \|f\|_{\mathcal{M}^{\Phi,\varphi}}.
\end{align*}
By taking the supremum of all $B$, we get the desired result.
\end{proof}

\begin{thm} (Weak version of Adams type result) \label{AdGulFrMaxOrlMorNecW}

Let $0< \a<n$, $\Phi\in \Delta^{\prime}$, $\varphi \in \Omega_{\Phi}$, $\beta\in(0,1)$, $\eta(t)\equiv\varphi(t)^{\beta}$ and $\Psi(t)\equiv\Phi(t^{1/\beta})$.

$1.~$ If $\varphi(t)$ satisfies \eqref{bounmax}, then the condition \eqref{eq3.6.V} is sufficient for the boundedness of $M_{\a}$ from $\mathcal{M}^{\Phi,\varphi}(\Rn)$ to $W\mathcal{M}^{\Psi,\eta}(\Rn)$.

$2.~$ If $\varphi\in{\mathcal{G}}_{\Phi}$, then the condition \eqref{condAdams}
is necessary for the boundedness of $M_{\a}$ from $\mathcal{M}^{\Phi,\varphi}(\Rn)$ to $W\mathcal{M}^{\Psi,\eta}(\Rn)$.

$3.~$ If $\varphi\in{\mathcal{G}}_{\Phi}$, then the condition \eqref{condAdams}
is necessary and sufficient for the boundedness of $M_{\a}$ from $\mathcal{M}^{\Phi,\varphi}(\Rn)$ to $W\mathcal{M}^{\Psi,\eta}(\Rn)$.
\end{thm}

\begin{proof}
The first part of the theorem is a corollary of Theorem \ref{AdGulFrMaxOrlMorW}.

We shall now prove the second part. Let $B_0=B(x_0,t_0)$ and $x\in B_0$. By Lemma \ref{estFrMax} we have $t_0^{\alpha}\leq C M_{\a} \chi_{B_0}(x)$. Therefore, by Lemmas \ref{charorlcw} and \ref{charOrlMor}
\begin{align*}
t_0^{\alpha}&\leq C\Psi^{-1}(|B_0|^{-1})\|M_{\a} \chi_{B_0}\|_{WL^{\Psi}(B_0)} \leq C\eta(t_0)\|M_{\a} \chi_{B_0}\|_{W\mathcal{M}^{\Psi,\eta}}
\\
&\leq C\eta(t_0)\|\chi_{B_0}\|_{\mathcal{M}^{\Phi,\varphi}}\leq C\frac{\eta(t_0)}{\varphi(t_0)}= C \varphi(t_0)^{\beta-1}.
\end{align*}
Since this is true for every $t_0>0$, we are done.

The third statement of the theorem follows from the first and second parts of the theorem.
\end{proof}
If we take $\Phi(t)=t^{p},\,p\in[1,\infty)$ and $\beta=\frac{p}{q}$ with $p < q<\infty$ at Theorem \ref{AdGulFrMaxOrlMorNecW} we get the following corollary which also can be seen as a special case of \cite[Theorem 2]{HakNakSaw}.
\begin{cor}\label{gmcoradw}
Let $1\le p < q<\infty$, $0< \a<n$ and $\varphi \in \Omega_{p}$.

$1.~$ If $\varphi(t)$ satisfies \eqref{eq4.6.GSMax}, then the condition
\eqref{eq3.6.VGM} is sufficient for the boundedness of $M_{\a}$ from $\mathcal{M}^{p,\varphi}(\Rn)$ to $W\mathcal{M}^{q,\varphi^{\frac{p}{q}}}(\Rn)$.

$2.~$ If $\varphi\in{\mathcal{G}}_{p}$, then the condition
\eqref{condad} is necessary for the boundedness of $M_{\a}$ from $\mathcal{M}^{p,\varphi}(\Rn)$ to $W\mathcal{M}^{q,\varphi^{\frac{p}{q}}}(\Rn)$.

$3.~$ If $\varphi\in{\mathcal{G}}_{p}$, then the condition \eqref{condad}
is necessary and sufficient for the boundedness of $M_{\a}$ from $\mathcal{M}^{p,\varphi}(\Rn)$ to $W\mathcal{M}^{q,\varphi^{\frac{p}{q}}}(\Rn)$.
\end{cor}

If we take $\varphi(t)=\frac{\Phi^{-1}\big(t^{-n}\big)}{\Phi^{-1}\big(t^{-\lambda}\big)}$, $0 \le \lambda \le n$, $\Psi(t)\equiv\Phi(t^{1/\beta})$, $\beta\in(0,1)$,
$$
\eta(t)\equiv\varphi(t)^{\beta}=\Big(\frac{\Phi^{-1}\big(t^{-n}\big)}{\Phi^{-1}\big(t^{-\lambda}\big)}\Big)^{\beta}
=\frac{\Psi^{-1}\big(t^{-n}\big)}{\Psi^{-1}\big(t^{-\lambda}\big)},
$$
at Theorem \ref{AdGulFrMaxOrlMorNecW} we get the following corollary.
\begin{cor}\label{MarchN6W}
Let $\Phi\in \Delta^{\prime}$, $\Psi(t)\equiv\Phi(t^{1/\beta})$ and $\beta\in(0,1)$. If the condition
\eqref{intcondAdamsOrlMorNw}
is satisfied, then the condition
\eqref{NScondAdamsOrlMorNw}
is necessary and sufficient for the boundedness of $M_{\a}$ from $\mathcal{M}^{\Phi,\lambda}(\Rn)$ to $W\mathcal{M}^{\Psi,\lambda}(\Rn)$.
\end{cor}

If we take into account Remark \ref{fat82} we get the following weak version of Adams result for Morrey spaces (see Theorem \ref{Adams1}).

\begin{cor}
Let $0< \alpha<n$, $1\le p < q<\infty$ and $0\le\lambda< n-\alpha p$. Then
$M_{\alpha}$ is bounded from $\mathcal{M}^{p,\lambda}(\Rn)$
to $W\mathcal{M}^{q,\lambda}(\Rn)$  if and only if $\frac{1}{p}-\frac{1}{q}=\frac{\alpha}{n-\lambda}$.
\end{cor}

{\bf Acknowledgements.} The research of V.S. Guliyev and F. Deringoz is partially supported by the grant of Ahi Evran University Scientific Research Project (FEF.A3.16.024).
The research of V.S. Guliyev is  partially supported by the grant of Presidium Azerbaijan National Academy of Science 2015. We thank the referee(s) for careful reading the paper and useful comments.

\end{document}